\documentclass[12pt,reqno]{amsart}

\usepackage{graphicx}
\usepackage{mathrsfs}
\usepackage{color}
\usepackage{xypic}
\usepackage{comment}
\usepackage{amssymb}
\usepackage{hyperref}

\usepackage[margin=1.5in]{geometry}

\newtheorem{theorem}{Theorem}
\newtheorem{lemma}[theorem]{Lemma}
\newtheorem{corollary}[theorem]{Corollary}

\newtheorem{proposition}[theorem]{Proposition}

\theoremstyle{definition}
\newtheorem{definition}[theorem]{Definition}
\newtheorem{example}{Example}

\theoremstyle{remark}
\newtheorem*{remark}{Remark}

\theoremstyle{definition}

\newcommand{\RR}{\mathbb{R}}

\newcommand{\ZZ}{\mathbb{Z}}

\DeclareMathOperator{\diam}{diam}

\begin{document}
	
	\title[Sufficiently connected PSC manifolds]{Classifying sufficiently connected PSC manifolds in $4$ and $5$ dimensions}
	\author[O. Chodosh]{Otis Chodosh}
	\address{OC: Department of Mathematics, Stanford University, Building 380, Stanford, California 94305, USA.}
	\email{ochodosh@stanford.edu}
	\author[C. Li]{Chao Li}
	\address{CL: Courant Institute of Mathematical Sciences
New York University, 251 Mercer St,
New York, NY 10012}
	\email{chaoli@nyu.edu}
	\author[Y. Liokumovich]{Yevgeny Liokumovich}
	\address{YL: Department of Mathematics, University of Toronto, 40 St
		George Street, Toronto, ON M5S 2E4, Canada.}
	\email{ylio@math.toronto.edu}

	\maketitle
	
	\begin{abstract} 
		We show that if $N$ is a closed manifold of dimension 
		$n=4$ (resp. $n=5$) with $\pi_2(N) = 0$ 
		(resp. $\pi_2(N)=\pi_3(N)=0$) that admits a metric of positive scalar curvature, then a finite cover $\hat N$ of $N$ is homotopy equivalent to $S^n$ or connected sums of $S^{n-1}\times S^1$. Our approach combines recent advances in the study of positive scalar curvature with a novel argument of Alpert--Balitskiy--Guth.
		
		Additionally, we prove a more general mapping version of this result. In particular, this implies that if $N$ is a closed manifold of dimensions $4$ or $5$, and $N$ admits a map of nonzero degree to a closed aspherical manifold, then $N$ does not admit any Riemannian metric with positive scalar curvature.
		
	\end{abstract}

	\section*{Introduction}
	
	We are concerned here with the problem of classification of manifolds admitting positive scalar curvature (PSC). For closed (compact, no boundary) $2$- and $3$-manifolds this problem is completely resolved, namely the sphere and projective plane are the only closed surfaces admitting positive scalar curvature and a $3$-manifold admits positive scalar curvature if and only if it has no aspherical factors in its prime decomposition. In particular, a $3$-manifold admitting positive scalar curvature has a finite cover diffeomorphic to $S^3$ or to a connected sum of finitely many $S^2\times S^1$. 
	
	The main result of this paper is the following partial generalization of this statement to dimensions $n=4,5$.

	\begin{theorem}\label{theo:main}
		Suppose that $N$ is a closed smooth $n$-manifold admitting a metric of positive scalar curvature and
		\begin{itemize}
			\item  $n=4$ and $\pi_2(N)=0$, or
			\item  $n=5$ and $\pi_2(N)=\pi_3(N) = 0$.
		\end{itemize}
		Then a finite cover $\hat N$ of $N$ is homotopy equivalent to $S^n$ or connected sums of $S^{n-1}\times S^1$. 
	\end{theorem}
	
	It was shown in \cite{ChodoshLi2020generalized,Gromov2020metrics} that if a closed $N^n$ is aspherical (i.e., $\pi_k(N)=0$ for all $k\geq 2$) and $n=4,5$  then there is no Riemannian metric of positive scalar curvature on $N$. Theorem \ref{theo:main} can thus be seen as a refinement of this into a positive result. 
	
	\begin{remark}
		By \cite[Theorem 1.3]{GadgilSeshadri2009topology} (see also \cite{Freedman,Milnor:hcobordism,KreckLuck2009rigidity}), we have that if $n=4$ and $\hat N$ is homotopy equivalent to $S^4$ or $S^3\times S^1$, or if $n=5$ (with no further restriction on the homotopy type), then homotopy equivalence in the conclusion to Theorem \ref{theo:main} can be upgraded to homeomorphism.
	\end{remark}

	We also prove a more general
	``mapping'' version of Theorem \ref{theo:main}.
	
	\begin{theorem} \label{theo:mapping}
		Suppose that $N$ is a closed smooth $n$-manifold with a metric of positive scalar curvature and
		there exists a non-zero degree map 
		$f:N \rightarrow X$, to a manifold $X$ satisfying
		\begin{itemize}
			\item  $n=4$ and $\pi_2(X)=0$, or
			\item  $n=5$ and $\pi_2(X)=\pi_3(X) = 0$.
		\end{itemize}
		Then a finite cover $\hat X$ of $X$ is homotopy equivalent to $S^n$ or connected sums of $S^{n-1}\times S^1$. 
	\end{theorem}

	We note that the following result immediately follows from Theorem \ref{theo:mapping}. 
	
	\begin{corollary}\label{coro.degree}
		Let $n\in \{4,5\}$, $X, N$ be closed oriented manifolds of dimension $n$, $X$ is aspherical. Suppose there exists a map $f: N\to X$ with $\deg f \ne 0$. Then $N$ does not admit any Riemannian metric of positive scalar curvature.
	\end{corollary}
	
	Recall that it was previously shown in \cite{ChodoshLi2020generalized,Gromov2020metrics} that closed aspherical (i.e., $\pi_k(N)=0$ for all $k\geq 2$) $n$-manifolds do not admit PSC for $n=4,5$. In \cite{Gromov2020metrics} a related statement was proven for manifolds admitting proper distance decreasing maps to uniformly contractible manifolds. In fact, Corollary \ref{coro.degree} seems to have been asserted by Gromov in the May 2021 version of his four lectures on scalar curvature \cite[p.\ 144-5]{gromov2019lectures}, but the (relatively simple) lifting argument does not appear there. 
	
	\subsection{Urysohn width bounds}
	
	Recall that a metric space $(X,d)$ has \emph{Urysohn $q$-width} $\leq \Lambda$ if there is a $q$-dimensional simplicial complex $K$ and a continuous map $X\to K$ so that $\diam f^{-1}(s) \leq \Lambda$ for all $s\in K$. As such, having finite Urysohn $q$-width implies that a manifold looks $\leq q$-dimensional in some macroscopic sense. 
	
	A well-known conjecture (cf.\ \cite[p.\ 63]{gromov2019lectures}) of Gromov posits that an $n$-manifold with scalar curvature $\geq 1$ has finite Urysohn $(n-2)$-width. Various forms of this conjecture are proven for $n=3$ \cite{GromovLawson,Katz1988diameter,MarquesNeves:width-psc,LiokumovichMaximo2020waist}, while the conjecture is largely open for $n\geq 4$ (some progress has been achieved in \cite{Boltov,BoltovDranishnikov}). 
	
	A key component in the proof of Theorem \ref{theo:main} is the following result.
	\begin{theorem}\label{thm:urysohn}
		For $(N^n,g)$ satisfying the hypothesis of Theorem \ref{theo:main}, the universal cover $(\tilde N,\tilde g)$ has finite Urysohn $1$-width. 
	\end{theorem}
	This follows by combining Corollary \ref{coro:filling-PSC} and Proposition \ref{prop:fill-to-Ur-w} below. A simple example where Theorem \ref{theo:main} applies is the product metric on $S^1\times S^3$ whose universal cover is $\mathbb{R}\times S^3$, clearly of finite Urysohn $1$-width. On the other hand, we note that the higher connectivity hypothesis in Theorem \ref{thm:urysohn} is necessary: compare with $T^2\times S^2$. 
	
	\begin{remark}
		Consider a metric $g_R$ on $S^3$ formed by capping off a cylinder $[-R,R] \times S^2(1)$ with hemispheres and smoothing out the resulting metric, so that the scalar curvature is $\geq 1$. The product metric $(S^1(1),g_S) \times (S^3,g_R)$ has scalar curvature $\geq1$ but the universal cover has Urysohn $1$-width $\sim R$. As such, the estimate in Theorem \ref{thm:urysohn} cannot be made quantitative (essentially, the issue is that the universal cover converges to $\mathbb{R}^2\times S^2(1)$ which has nontrivial $\pi_2$). 
	\end{remark}
	
	As we were finishing this paper, we discovered that in his recently updated (May 2021) version of his four lectures on scalar curvature, Gromov has indicated a proof of the classification of PSC $3$-manifolds  \cite[p.\ 135]{gromov2019lectures} by using finiteness of the $1$-Ursyohn width of the universal cover. Our proof of Theorem \ref{theo:main} follows a similar strategy once Theorem \ref{thm:urysohn} is proven.
	
	\subsection{Remarks on positive isotropic curvature} Theorem \ref{theo:main} has an interesting relationship to well-known conjectures of Gromov \cite[Section 3, (b)]{gromov1996positive} and Schoen \cite{Schoen2007conjecture} concerning the topology of closed $n$-manifolds admitting a metric with positive isotropic curvature (PIC). Namely, they (respectively) conjecture that if a closed manifold has a PIC metric then the fundamental group is virtually free and a finite cover is diffeomorphic to either a sphere or connect sums of finitely many $S^{1}\times S^{n-1}$.

	There have been distinct approaches to such a question, relying on either minimal surface theory or Ricci flow. Using minimal surface theory, Micallef--Moore have shown that if $M^n$ is a closed PIC-manifold then $\pi_k(M) = 0$ for $k=2,\dots,[\tfrac n2]$ \cite{MM}. In particular, if $M$ is simply connected, then it is homeomorphic to a sphere. In a related work, Fraser has proven that an $n$-manifold ($n\geq 5$) with PIC does not contain a subgroup isomorphic to $\ZZ\oplus \ZZ$ \cite{Fraser}. 
	
	On the other hand, using Ricci flow, Hamilton has classified $4$-manifolds admitting PIC that do not contain nontrivial incompressible $(n-1)$-dimensional space forms \cite{Hamilton:PIC-RF}. This was extended to prove the Gromov--Schoen conjectures for $n=4$ in \cite{ChenTangZhu:PIC-4d}. In higher dimensions, Brendle--Schoen \cite{BrendleSchoen2009spheretheorem} and Nguyen \cite{Nguyen} proved the PIC condition is preserved under the Ricci flow; this is an important ingredient in Brendle--Schoen's proof of the differentiable sphere theorem. Recently, Brendle has achieved a breakthrough in the study of the Ricci flow of PIC-manifolds and has extended Hamilton's result to dimensions $n\geq 12$ \cite{Brendle:PIC-RF}; as above, this result has been used to prove the Gromov--Schoen conjectures for $n\geq 12$ \cite{Huang:PIC-nd}. 
	
	We note that since PIC implies PSC, combining \cite{MM} with Theorem \ref{theo:main} yields an alternative proof of Gromov's conjecture (the fundamental group is virtually free) for $n=4$ and proves a weak version of Schoen's conjecture for $n=4$ (i.e., with homotopy equivalence replacing diffeomorphism). Furthermore, Theorem \ref{theo:main} implies that a PIC $5$-manifold with $\pi_3(M) = 0$ satisfies Gromov's conjecture and the same weak version of Schoen's conjecture. It is an interesting question if a $5$-manifold with PIC has $\pi_3(M)=0$ (note that $\pi_2(M)=0$ by \cite{MM}).

	\subsection{Organization of the paper}
	In Section \ref{sec:filling} we revisit the filling radius estimates from \cite{ChodoshLi2020generalized,Gromov2020metrics}. In Section \ref{sec:fill-vs-Ur} we show that such estimates imply Theorem \ref{thm:urysohn}. Then, we complete the proof of Theorem \ref{theo:main} in Section \ref{section:homotopy.type}. Finally, in Section \ref{sec:mapping} we prove Theorem \ref{theo:mapping}.

	\subsection*{Acknowledgements}
	O.C. was supported by a Sloan Fellowship, a Terman Fellowship, and NSF grant DMS-2016403. C.L. was supported by NSF
	grant DMS-2005287. Y.L. was supported by a NSERC Discovery grant and NSERC Accelerator Award. We are grateful to Hannah Alpert, Larry Guth, Ciprian Manolescu, Kasra Rafi, and Boyu Zhang for various discussions related to this article. We would like to thank
	Georg Frenck for suggesting the statement 
	of Theorem \ref{theo:mapping} to us as well as an anonymous referee for several helpful suggestions.

	\section{Filling estimates}\label{sec:filling}
	
	In \cite{ChodoshLi2020generalized,Gromov2020metrics} it was shown that a closed aspherical $n$-manifold does not admit positive scalar curvature for $n=4,5$ by combining a linking argument with a filling radius inequality in the presence of positive scalar curvature. In this section we observe that this filling radius inequality carries over to the setting considered here. 
	
	We begin by summarizing the results contained in \cite{ChodoshLi2020generalized} that will be needed in this paper. 
	\begin{theorem}\label{thm:statement-aspherical} Consider $(N^n,g)$ a closed Riemannian $n$-manifold with scalar curvature $R\geq 1$. Fix a Riemannian cover $(\hat N,\hat g)$.
	    \begin{enumerate}
	        \item Suppose that $n=4$. There is a universal constant $L_0>0$ with the following property. Consider a closed embedded $2$-dimensional submanifold  $\hat \Sigma_2 \subset \hat N$ with $[\hat\Sigma_2] = 0 \in H_2(\hat N;\ZZ)$. Then there is a $3$-chain $\hat\Sigma_3'\subset B_{L_0}(\hat\Sigma_2)$ and a closed embedded $2$-dimensional submanifold $\hat\Sigma_2'$ with
	        \[\partial\hat\Sigma_3'=\hat\Sigma_2 - \hat\Sigma_2'\]
	        as chains, such that for every connected component $S$
	        of $\hat\Sigma_2'$
	        the extrinsic diameter of $S$ satisfies %$\diam(\hat\Sigma_2') \leq L_0$. 
	        $\diam(S) \leq L_0$.
	        \item Suppose that $n=5$. There is a universal constant $L_0>0$ with the following property. Consider a closed embedded $3$-dimensional submanifold  $\hat\Sigma_3 \subset \hat N$ with $[\hat\Sigma_3] = 0 \in H_3(\hat N;\ZZ)$. Then there is a $4$-chain $\hat \Sigma_4'\subset B_{L_0}(\hat\Sigma_3)$ and a closed embedded $3$-dimensional submanifold $\hat\Sigma_3'$ with
	        \[\partial\Sigma_4' = \Sigma_3 - \Sigma_3'\]
	        as chains, as well as $3$-chains $\hat U_1,\dots,\hat U_m$ with $\diam(\hat U_j)\leq L_0$ and $2$-cycles $\{\hat \Gamma_j^l : j =1,\dots,m, l=1,\dots,k(j)\}$ with $\diam(\hat \Gamma_j^l) \leq L_0$ and so that
	        \[
	        \hat\Sigma_3 = \sum_{j=1}^m \hat U_j
	        \]
	        and
	        \[
	        \partial \hat U_j = \sum_{l=1}^{k(j)} \hat\Gamma_j^l,
	        \]
	        for each $j=1,\dots,m$, where both equalities hold as chains (not just in homology).  Finally, there is an integer $q$ and a function 
	        \[
	        u : \{(j,l) : j =1,\dots,m, l=1,\dots,k(j)\} \to \{1,\dots,q\}
	        \]
	        so that for $r \in \{1,\dots,q\}$, we have
	        \[
	        \diam\left(\bigcup_{(j,l) \in u^{-1}(r)} \hat\Gamma_j^l \right) \leq L_0
	        \]
	        and moreover we have
	        \[
	        \sum_{(j,l) \in u^{-1}(r)} \hat\Gamma_j^l = 0
	        \]
	        as $2$-chains, for $r \in \{1,\dots,q\}$. 
	    \end{enumerate}
	\end{theorem}
	\begin{proof}
	    When $n=4$, one can solve Plateau's problem to find $\hat\Sigma_3$ minimizing area with $\partial\hat\Sigma_3 = \hat\Sigma_2$. Applying the ``$\mu$-bubble technique'' (cf.\ \cite[\S 3]{ChodoshLi2020generalized}), we can find $\hat\Sigma_2' \subset \hat \Sigma_3$ with $d_{\hat\Sigma_3}(\hat\Sigma_2',\hat\Sigma_2) \leq L_0$ and so that $\hat\Sigma_2' \subset \hat\Sigma_3$ is a ``stable $\mu$-bubble'' in the sense of \cite[Lemma 14]{ChodoshLi2020generalized}. By \cite[Lemma 16]{ChodoshLi2020generalized}, the intrinsic diameter of each component is $\leq L_0$ (taking $L_0$ larger if necessary). This proves the assertion (since extrinsic distances are bounded by the intrinsic distances).
	    
	    Similarly, when $n=5$, we can solve Plateau's problem to find $\hat\Sigma_4$ minimizing area with $\partial\hat \Sigma_4= \hat\Sigma_3$. As before, we can find a ``stable $\mu$-bubble'' $\hat \Sigma_3'$ with $d_{\hat\Sigma_4}(\hat\Sigma_3',\hat\Sigma_3) \leq L_0$. Finally, the construction of the $\hat U_j,\hat \Gamma_j^k$ follows from the ``slice-and-dice'' procedure from \cite[\S 6.3-6.4]{ChodoshLi2020generalized}. 
	    
	    Note that the last conclusion (i.e., that $ \sum_{(j,l) \in u^{-1}(r)} \hat\Gamma_j^l = 0$) was stated slightly differently in \cite{ChodoshLi2020generalized}. To be precise, it was proven that the cycles $\sum_{(j,l) \in u^{-1}(r)} \hat\Gamma_j^l$ are disjoint for distinct $r$ (cf.\ \cite[\S 6.4]{ChodoshLi2020generalized}). Now, by using $\sum_{r=1}^q \sum_{(j,l) \in u^{-1}(r)} \hat\Gamma_j^l = \partial(\sum_{j=1}^m U_j) = 0$, we find that each term in the sum must vanish. 
	\end{proof}

    \begin{figure}[h!]
            \centering
            \includegraphics[width=\textwidth]{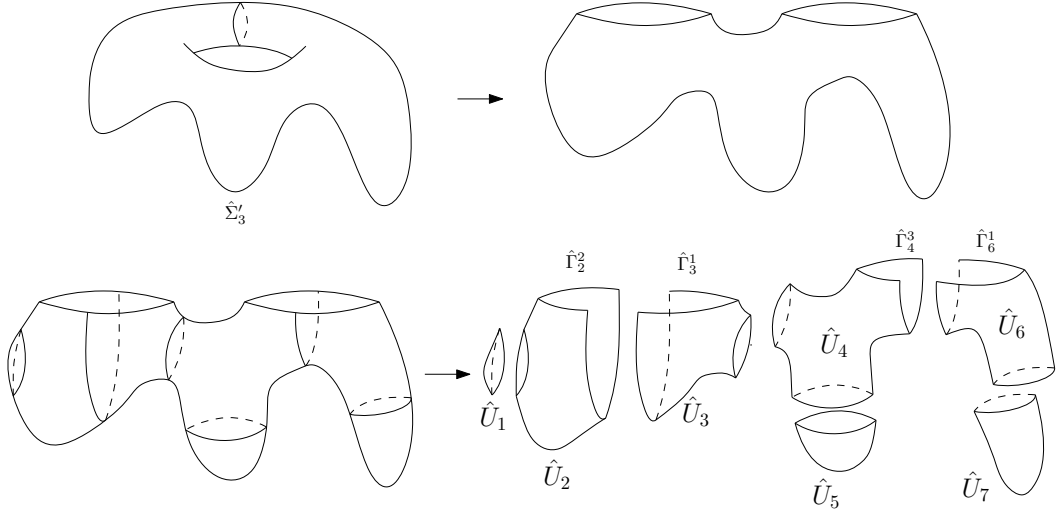}
            \caption{Cutting the 3-cycle $\hat \Sigma_3'$ into small pieces.}
            \label{fig:slice_dice}
    \end{figure}

    \begin{example}\label{example:slice_dice}
        We illustrate the ``slice-and-dice'' procedure and its relevance to the statements of Theorem \ref{thm:statement-aspherical} with figure \ref{fig:slice_dice}, where $\hat\Sigma_3'$ is diffeomorphic to $S^2\times S^1$. We first cut (slice) $\hat\Sigma_3'$ by an embedded $S^2$ and view the result as a $3$-manifold with boundary, which we further cut (dice) into seven $3$-chains $\hat U_1,\cdots, \hat U_7$, such that each $\hat U_j$ satisfies $\diam \hat U_j\le L_0$ (of course, the number of chains may vary in different examples). We label the boundary components of $\hat U_j$, from left to right in the figure, by $\hat \Gamma_j^l$, $l=1,\cdots, k(j)$. Note that in this case, there are $4$ such boundary components that are non-smooth, namely $\hat \Gamma_2^2, \hat \Gamma_3^1, \hat\Gamma_4^3, \hat \Gamma_6^1$. The function $u$ groups different $\hat\Gamma_j^l$ that glue together into a $2$-cycle. For example, we have
        \[u(2,2)=u(3,1)=u(4,3)=u(6,1),\]
        and
        \[u(1,1)=u(2,1), \quad u(3,2)=u(4,1),\quad u(4,2)=u(5,1),\quad u(6,2)=u(7,1).\]
        Moreover, the value of $u$ on different groups of $\hat\Gamma_j^l$ are different. (e.g. $u(2,2)\ne u(1,1)$.) Note here that $\sum_{(j,l)\in u^{-1}(r)} \hat \Gamma_j^l=0$ for each $r$.

    \end{example}
    
    The following proposition will be used to replace \cite[Proposition 10]{ChodoshLi2020generalized} in the more general setting considered here. 
	
	\begin{proposition}\label{prop:FF-iso-cover}
		Consider $\pi :(\hat N,\hat g) \to (N,g)$ a regular\footnote{Recall that a cover is regular if the group of deck transformations acts transitively on the fibers. In particular, the universal cover is a regular cover.} Riemannian covering map of $n$-dimensional manifolds, with $(N,g)$ compact. Assume that $H_\ell(\hat N,\mathbb{Z}) = 0$. Then for $r>0$ there is $R=R(r) < \infty$ so that for any $x \in \hat N$, $H_\ell(B_r(x),\mathbb{Z})\to H_\ell(B_R(x),\mathbb{Z})$ is the zero map.
	\end{proposition}
	\begin{proof}
		We first fix $x=x_0$. For any $r>0$ there is $r_1\in [r,2r]$ with $\overline{B_{r_1}(x_0)}$ a compact manifold (with boundary). By Corollaries A.8 and A.9 in \cite{Hatcher2002AlgebraicTopology}, the homology groups of $\overline{B_{r_1}(x_0)}$ are finitely generated. Assume that $\alpha_1,\dots,\alpha_J$ generates $H_\ell(\overline{B_{r_1}(x_0)},\mathbb{Z})$. By assumption, each $\alpha_i = \partial\beta_i$ for some $(\ell+1)$-chains $\beta_1,\dots,\beta_J$. Choose $R_1=R_1(r)$ so that $\beta_i \in B_{R_1}(x_0)$ for $i=1,\dots,J$. Then, we see that $H_\ell(\overline{B_{r_1}(x_0)},\mathbb{Z}) \to H_\ell(B_{R_1}(x_0),\mathbb{Z})$ is the zero map, so in particular 
		\[
		H_\ell(B_r(x_0),\mathbb{Z}) \to H_\ell(B_{R_1}(x_0),\mathbb{Z})
		\]
		is the zero map. 
		
		Now, for any $x \in \hat N$, we can assume (using a deck transformation) that $d(x,x_0)\leq \diam N$. Thus, 
		\[
		B_r(x) \subset B_{r+\diam N}(x_0)
		\]
		and
		\[
		B_{R_1(r+\diam N)}(x_0) \subset B_{R_1(r+\diam N)+\diam N}(x)
		\]
		Thus, we find that the assertion holds for $R(r) = R_1(r+\diam N) + \diam N$. 
	\end{proof}
	
%	Now, we observe that the techniques used in \cite{ChodoshLi2020generalized,Gromov2020metrics} yields the following filling radius estimate. 
    Putting these facts together, we thus obtain the following generalization of the filling estimate obtained in \cite{ChodoshLi2020generalized,Gromov2020metrics}. 
	
	\begin{corollary}\label{coro:filling-PSC}
		Suppose that for $n\in\{4,5\}$, $(N^n,g)$ is a closed Riemannian $n$-manifold with positive scalar curvature
		and $\pi_2(N) = \dots = \pi_{n-2}(N) = 0$. Then there is $L=L(N,g)>0$ with the following property.
		Consider $\Sigma_{n-2}$ an closed embedded $(n-2)$-submanifold in $\tilde N$ the universal cover. Then $\Sigma_{n-2}$ is nullhomologous in $B_L(\Sigma_{n-2})$. 
	\end{corollary}
	\begin{proof}
	    Observe that universal cover $\tilde N$ has $\pi_1(\tilde N) = \dots = \pi_{n-2}(\tilde N) = 0$. By the Hurewicz theorem, $H_{n-3}(\tilde N,\mathbb{Z}) = H_{n-2}(\tilde N,\mathbb{Z})=0$.
	
	    When $n=4$ the assertion immediately follows from a combination of Theorem \ref{thm:statement-aspherical} with Proposition \ref{prop:FF-iso-cover}. Indeed, Theorem \ref{thm:statement-aspherical} implies that $\Sigma_{2}$ is homologous to $\Sigma_2'$ in $B_{L_0}(\Sigma_2)$ where $\diam(\Sigma_2')\leq L_0$. Proposition \ref{prop:FF-iso-cover} implies that $\Sigma_2'$ can be filled in a $R(L_0)$-neighborhood. Thus, $\Sigma_2$ can be filled in a $(L_0+R(L_0))$-neighborhood.
	    
	    When $n=5$ the proof is more complicated due to the nature of the ``slice-and-dice'' decomposition in Theorem \ref{thm:statement-aspherical}. Fix $\hat \Sigma_3'\subset B_{L_0}(\Sigma_3)$ homologous to $\Sigma_3$ and $\{\hat U_j\}$ and $\{\hat \Gamma_j^l\}$ with the properties described in Theorem \ref{thm:statement-aspherical}. We can now fill $\hat\Sigma_3'$ in a bounded neighborhood following \cite[\S 6.4]{ChodoshLi2020generalized}, which we explain here. Since $\diam(\hat \Gamma_j^l) \leq L_0$, Proposition \ref{prop:FF-iso-cover} implies that each $\hat \Gamma_j^l = \partial \tilde \Gamma_j^l$ for a $3$-chain with $\diam(\tilde\Gamma_j^l) \leq R(L_0)$. Then, because $\diam(\hat U_j) \leq L_0$, 
	    \[
	    \hat U_j - \sum_{l=1}^{k(j)} \tilde \Gamma_j^l
	    \]
	    is a $3$-cycle of diameter $\leq L_0 + 2R(L_0)$. Thus, by Proposition \ref{prop:FF-iso-cover}, there is a $4$-chain $\tilde U_j$ with $\diam(\tilde U_j) \leq R(L_0 + 2R(L_0))$ and
	    \[
	    \partial\tilde U_j = \hat U_j - \sum_{l=1}^{k(j)} \tilde \Gamma_j^l.
	    \]
	    
	    On the other hand, as was proven in Theorem \ref{thm:statement-aspherical}, there is $u:\{(j,l) : j=1,\dots,m,l=1,\dots,k(j)\}\to\{1,\dots,q\}$ so that
	    \[
	    \diam\left( \bigcup_{(j,l) \in u^{-1}(r)} \hat \Gamma_j^l  \right) \leq L_0
	    \]
	    and
	    \[
	    \sum_{(j,l) \in u^{-1}(r)} \hat \Gamma_j^l = 0
	    \]
	    as $2$-chains.

	   % In particular,
	 %   \[
	%    \hat\Sigma_3' - \sum_{j=1}^m \partial \tilde U_j = \sum_{j=1}^m \sum_{l=1}^{k(j)} \tilde \Gamma_j^l = \sum_{r=1}^q \sum_{(j,l) \in u^{-1}(r)} \tilde \Gamma_j^l
	 %   \]

	    %Because the terms $\sum_{(j,l) \in u^{-1}(r)} \tilde \Gamma_j^l$ have pairwise disjoint boundaries and sum to a cycle, we see that 
	    
	    As such, for $r\in\{1,\dots,q\}$, $\sum_{(j,l) \in u^{-1}(r)} \tilde \Gamma_j^l$ is a $3$-cycle (of diameter bounded by $2R(L_0)+L_0$ and thus there is a $4$-chain $\hat \Xi_r$ with $\diam(\Xi_r) \leq R(L_0 + 2R(L_0))$ and 
	    \[
	    \partial \Xi_r = \sum_{(j,l) \in u^{-1}(r)} \tilde \Gamma_j^l. 
	    \]
	    This yields
	    \[
	    \hat\Sigma_3' = \partial\left[\sum_{r=1}^q \Xi_q + \sum_{j=1}^m \tilde U_j \right]
	    \]
	    with
	    \[
	    \sum_{r=1}^q \Xi_q + \sum_{j=1}^m \tilde U_j \subset B_{R(L_0+2R(L_0))}(\hat\Sigma_3').
	    \]
	    Thus, $\Sigma_3$ is null homologous in a $(R(L_0+2R(L_0)) + R(L_0))$-neighborhood. This completes the proof.
\end{proof}
	
	\begin{figure}[h!]
	    \centering
	    \includegraphics[width=\textwidth]{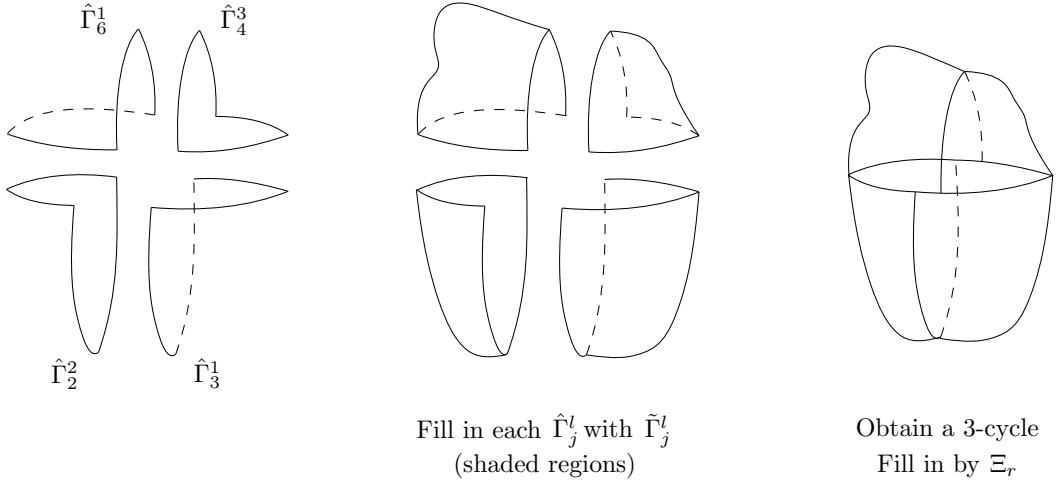}
	    \caption{Fill in $\left\{\hat\Gamma_j^l: u(j,l)=r\right\}$.}
	    \label{fig:fill_in}
	\end{figure}

	\begin{example}
	    Continuing Example \ref{example:slice_dice}, we illustrate in Figure \ref{fig:fill_in} how Corollary \ref{coro:filling-PSC} works for $\hat \Sigma_3'$ in Figure \ref{fig:slice_dice}. Consider all $2$-cycles $\hat \Gamma_j^i$ with $u(j,i)=r$. Fill in $\hat\Gamma_j^l$ with a 3-chain $\tilde \Gamma_j^l$.
	    %(shaded regions in Figure \ref{fig:fill_in}).
	    By construction, the sum of these $\tilde \Gamma_j^l$ forms a $3$-cycle, which can then be filled in by a 4-chain $\Xi_r$. By Proposition \ref{prop:FF-iso-cover} and Corollary \ref{coro:filling-PSC}, the diameter of all these fill-ins are bounded by $R(L_0+2R(L_0))$.
	\end{example}

	\section{Filling versus Urysohn width}\label{sec:fill-vs-Ur}
	The next result is inspired by the work of Hannah Alpert, Alexei Balitsky and Larry Guth \cite{AlBaGu},
	which we learnt about from a talk by Hannah Alpert. The strategy should be compared with \cite[Corollary 10.11]{GromovLawson}. 
	\begin{proposition}\label{prop:fill-to-Ur-w}
		Assume that $(N^n,g)$ has the property that any closed embedded $(n-2)$-submanifold in the universal cover $\Sigma_{n-2}\subset \tilde N$ can be filled in $B_L(\Sigma_{n-2})$. Then the universal cover $(\tilde N,\tilde g)$ satisfies:
		\begin{multline}\label{eq.urysohn}
		\tag{$*$}\text{for any point $p \in \tilde N$, each connected component of} \\\text{a level set of $d(p,\cdot)$ has diameter $\leq 20L$.} 
		\end{multline}
	\end{proposition}
	Note that Corollary \ref{coro:filling-PSC} implies a manifold $(N,g)$ in Theorem \ref{theo:main} satisfies the assumptions of Proposition \ref{prop:fill-to-Ur-w}. By the argument in \cite[Corollary 10.11]{GL:complete}, this shows that the universal cover $(\tilde N,\tilde g)$ has Urysohn 1-width $\le 20L$. In particular, the macroscopic dimension of $\tilde N$ is $1$.
	
	\begin{proof}Let $p \in \tilde{N}$ be a point and consider level sets of the distance function 
		$f(x) = d(p, x)$. 
		
		For the sake of contradiction, suppose that there is a curve $\gamma \subset f^{-1}(t)$ connecting points $x,y$ with $d(x,y) \geq 20L$. Fix a minimizing geodesic $\eta_x$ from $p$ to $x$ (and similarly for $\eta_y$) and consider the triangle $T = \eta_x * \gamma * -\eta_y$. Fix $0<\ell < L$ so that $\partial B_{4L+\ell}(x)$ and $\partial B_{L+\ell}(\eta_x)$ are smooth hypersurfaces intersecting transversely. Set $\Sigma_{n-2} : = \partial B_{4L+\ell}(x) \cap \partial B_{L+\ell}(\eta_x)$. Note that we have not ruled out $\Sigma_{n-2} =\emptyset$; in this case we will take $d(\Sigma_{n-2},\cdot) = \infty$ below.
		%(Note that as a consequence of Lemma \ref{lem:dist-sigma-gamma} below we see that $\Sigma_{n-2}\neq \emptyset$.) 
		
		By construction,
		\[
		d(\Sigma_{n-2},\eta_x) > L.
		\]
		Set $\Sigma_{n-1}' : = \partial B_{4L+\ell}(x) \cap \overline{B_{L+\ell}(\eta_x)}$ and note that $\partial \Sigma_{n-1}' =\Sigma_{n-2}$. Observe that since $\eta_x$ is a minimizing geodesic between $p$ and $x$ and $\Sigma_{n-1}'$ is a subset of $\partial B_{4L+\ell}(x)$, it must hold that $\eta_x$ intersects $\Sigma_{n-1}'$ exactly once and does so orthogonally (and thus transversally). We will return to this observation below. 
		
		\begin{lemma}\label{lem:dist-sigma-gamma}
		$d(\Sigma_{n-2},\gamma) \geq d(\Sigma_{n-1}',\gamma) > L$.
		\end{lemma}
		\begin{proof}
		    We first prove that $d(\Sigma_{n-1}',\gamma) > L$. Choose $s \in \Sigma_{n-1}'$ with $d(s,\gamma) = d(\Sigma_{n-1}',\gamma)$. There is $e \in \eta_x$ so that $d(s,e) \leq L+\ell$. We have
			\[
			d(x,e) \geq d(x,s) - d(s,e) \geq 4L +\ell - (L+\ell) = 3L. 
			\]
			Since $\eta_x$ is minimizing (and has length $t$), we have $d(p,e) \leq t-3L$. Thus,
			\[
			d(p,s) \leq d(p,e) + d(e,s) \leq t-3L + L + \ell = t-2L + \ell. 
			\]
			Thus,
			\[
			d(s,\gamma) \geq d(p,\gamma) - d(p,s) \geq t - (t-2L+\ell) = 2L - \ell.
			\]
			This completes the proof of $d(\Sigma_{n-1}',\gamma) > L$. Since $\Sigma_{n-2} \subset \Sigma_{n-1}'$ it clearly holds that $d(\Sigma_{n-2},\gamma) \geq d(\Sigma_{n-1}',\gamma)$. 
			
		\end{proof}
		\begin{lemma}\label{lem:disjoint-sigma3-etay}
			$\Sigma_{n-1}' \cap \eta_y = \emptyset$.
		\end{lemma}
		\begin{proof}
			Suppose the contrary. Consider $s \in \Sigma_{n-1}' \cap \eta_y$. Note that $d(s,x) = 4L+\ell$ and there is $e \in \eta_x$ with $d(s,e)\leq L+\ell$. We have
			\[
			d(x,e) \leq d(x,s) + d(e,s)  \leq 5 L + 2\ell. 
			\]
			As such,
			\[
			d(p,e) \geq t-5L - 2\ell,
			\]
			so
			\[
			d(p,s) \geq d(p,e) - d(e,s) \geq t - 5L - 2\ell - L - \ell = t-6L -3\ell.
			\]
			Thus,
			\[
			d(s,y) \leq 6L + 3\ell. 
			\]
			However, this contradicts
			\[
			20 L \leq d(x,y) \leq d(x,s) + d(s,y) \leq 4L + \ell + 6L +3\ell = 10L + 4\ell.  
			\]
			This completes the proof. 
		\end{proof}
		\begin{lemma}\label{dist:sigma-etay}
			$d(\Sigma_{n-2},\eta_y) > L$.  
		\end{lemma}
		\begin{proof}
			The proof is similar to the previous argument. Suppose we have $s \in \Sigma_{n-2}$ and $e_y \in\eta_y$ with $d(s,e_y) \leq L$. There is $e_x \in \eta_x$ with $d(s,e_x)=L+\ell$. Note that $d(s,x) = 4L+\ell$. Thus,
			\[
			d(p,e_x) = t - d(x,e_x) \geq t - d(x,s) - d(s,e_x) \geq t-  5L - 2\ell.
			\]
			Thus,
			\[
			d(p,e_y) \geq d(p,e_x) - d(e_x,e_y) \geq d(p,e_x) - d(e_x,s) - d(s,e_y) = t - 7L - 3\ell .
			\]
			This implies that 
			\[
			d(y,e_y) \leq 7L + 3\ell. 
			\]
			However, this contradicts
			\[
			20 L \leq d(x,y) \leq d(x,s) + d(s,e_y) + d(e_y,y) \leq 12L + 4\ell.
			\]
			This completes the proof. 
		\end{proof}

		We can now complete the proof of Proposition \ref{prop:fill-to-Ur-w}. Perturb the triangle $T$ to be a smooth embedded curve $T'$ still intersecting $\Sigma_{n-1}'$ transversely. As long as the perturbation is small, $T'\cap \Sigma_{n-1}'$ will consist of a single point (thanks to Lemmas \ref{lem:dist-sigma-gamma} and \ref{lem:disjoint-sigma3-etay}, along with the observation that $\eta_x$ intersects $\Sigma_{n-1}'$ transversely in exactly one point). Assume first that $\Sigma_{n-2}\neq \emptyset$. By assumption, there is $\Sigma_{n-1} \subset B_L(\Sigma_{n-2})$ with $\partial\Sigma_{n-1} = \Sigma_{n-2}$. Using Lemmas \ref{lem:dist-sigma-gamma} and \ref{dist:sigma-etay} as well as $d(\Sigma_{n-2},\eta_x) = L + \ell$ we find that $\Sigma_{n-1} \cap T' = \emptyset$. As such, $T'$ has nontrivial algebraic intersection with the cycle $\Sigma_{n-1}'-\Sigma_{n-1}$. This is a contradiction since $\tilde N$ is simply connected. 
		
		If $\Sigma_{n-2} = \emptyset$, then the argument is similar but simpler. In this case, we note that $\Sigma_{n-1}'$ is a cycle and combining Lemmas \ref{lem:dist-sigma-gamma} and \ref{lem:disjoint-sigma3-etay} with the fact that  $\Sigma_{n-1}'$ intersects $\eta_x$ transversely exactly once, we see that $\Sigma_{n-1}'$ is a cycle with nontrivial algebraic intersection with $T'$, a contradiction as before. 
	\end{proof}

	\section{Fundamental group and homotopy type}\label{section:homotopy.type}
	
	In this section, we prove Theorem \ref{theo:main}. We first prove (see Corollary \ref{coro:virtually-free} below) that a manifold $(N^n,g)$ whose universal cover satisfies the conclusion of Proposition \ref{prop:fill-to-Ur-w} has virtually free fundamental group. (Recall that a group is virtually free if it processes a free subgroup of finite index.) This fact seems to be well-known among certain experts (in particular, see p.\ 135 in the May 2021 version of Gromov's four lectures on scalar curvature \cite{gromov2019lectures}). We give a proof here, roughly following the strategy used in \cite{RamachandranWolfson2010fillradius}. The argument is based on notion of the number of \textit{ends} of a group.
	
	\begin{definition}
		Given a group $G$, its number of ends, $e(G)$, is defined as the number of topological ends of $\tilde K$, where $\tilde K\to K$ is a regular covering of finite simplicial complexes $K,\tilde K$, and $G$ is the group of deck transformations.
	\end{definition}
	
	It follows from \cite{Epstein1962ends} that a finitely generated group can have $0,1,2$ or infinitely many ends. Our main result here is as follows:
	
	\begin{proposition}\label{prop:group.ends}
		Suppose $(N,g)$ is a closed Riemannian manifold satisfying 
		%    \eqref{eq.urysohn}
		conclusions of Proposition \ref{prop:fill-to-Ur-w}. Then any finitely generated subgroup $G$ of $\pi_1(N)$ cannot have one end.
	\end{proposition}
	
	We will prove this below, but first we note that it yields the desired statement:
	
	\begin{corollary}\label{coro:virtually-free}
		Suppose $(N,g)$ is a closed Riemannian manifold satisfying %\eqref{eq.urysohn}
		conclusions of Proposition \ref{prop:fill-to-Ur-w}. Then, $\pi_1(N)$ is virtually free.
	\end{corollary}
	\begin{proof}
		We follow the proof of \cite[Theorem 2.5]{RamachandranWolfson2010fillradius}. Indeed, by combining the main result of \cite{Dunwoody} (cf.\ \cite[\S 7]{ScottWall}) with 
		%Corollary \ref{coro:virtually-free}, 
		Proposition \ref{prop:group.ends},
		$\pi_1(N)$ is the fundamental group of a finite graph of groups with finite edge and vertex groups. The assertion now follows from Proposition 11 in Chapter II, Section 2.6 of \cite{Serre} (or e.g., \cite[Theorem 7.3]{ScottWall}). 
	\end{proof}
	
	Moreover, we observe that given these results, we can finish the proof of Theorem \ref{theo:main}.
	\begin{proof}[Proof of Theorem \ref{theo:main}]
		By Corollary \ref{coro:filling-PSC}, Proposition \ref{prop:fill-to-Ur-w}, and Corollary \ref{coro:virtually-free}, $\pi_1(N)$ is virtually free. Let $G\subset \pi_1(N)$ be a finite index subgroup which is a free group. Consider the finite covering $\hat N\xrightarrow{\hat p} N$ such that the image of $\hat p_\#$ is $G$. Then $\pi_1(\hat N)$ is a finitely generated free group. Since $\pi_2(\hat N)=\cdots=\pi_{n-2}(\hat N)=0$, \cite[Section 2 and Section 3]{GadgilSeshadri2009topology} implies that $\hat N$ is homotopy equivalent to $S^n$ or connected sums of $S^{n-1}\times S^1$.
	\end{proof}

	We now give the proof of Proposition \ref{prop:group.ends}:
	\begin{proof}[Proof of Proposition \ref{prop:group.ends}]
		Suppose there is a finitely generated subgroup $G$ of $\pi_1(N)$ with one end. We will show that this leads to a contradiction. 
		
		We divide the proof into several steps. Take a cover $N_0 \xrightarrow{p} N$ such that $p_\#(\pi_1(N_0)) = G$. Because $p_\# : \pi_1(N_0) \to \pi_1(N)$ is injective this ensures that $\pi_1(N_0) = G$. If $G$ is finite then $e(G)=0$, so we can assume $G$ is infinite.
		
		Since $G$ is finitely generated, we can find $K \subset N_0$ a compact submanifold with boundary containing representatives of all of the generators of $G$. Write $i : K\to N_0$ for the inclusion map and note that $i_\#: \pi_1(K) \to \pi_1(N_0) = G$ is surjective. Let $H = \ker i_\#$ so $G = \pi_1(K) /H$. Choose $j : \tilde K\to K$ the cover (with $\tilde K$ path connected) of $K$ so that $j_\#(\pi_1(\tilde K)) = H$. Since $H$ is a normal subgroup of $\pi_1(K)$, the covering $j: \tilde K\to K$ is regular, and the group of deck transformations of $j$ is isomorphic to $\pi_1(K)/H = G$. Thus $\tilde K$ is noncompact. Note that $j_\#\circ i_\#: \pi_1(\tilde K)\to \pi_1(N_0)$ is the zero map, so we can lift $i$ to $\tilde i : \tilde K \to \tilde N$, where $\tilde N$ is the universal cover of $N$. (We emphasize that $\tilde K$ is not necessarily the universal cover of $K$.)

		As such, we have the following diagram of spaces:
		\[
		\xymatrix{
			\tilde K \ar[r]^{\tilde i} 
			\ar[d]^j & \tilde N  \ar[d]^{\tilde p}\\
			K \ar[r]^i & N_0\ar[d]^p\\
			& N
		}
		\]
		The maps $i,\tilde i$ are inclusions of codimension zero submanifolds with boundary; indeed:
		\begin{lemma}\label{lemm:tilde.i.proper}
			$\tilde i$ is a proper embedding
		\end{lemma}
		\begin{proof}
			We first show that $\tilde i$ is injective. Suppose that $\tilde i(\tilde a) = \tilde i(\tilde b)$. Connect $a,b$ by a curve $\tilde \eta$ in $\tilde K$. By assumption, $\tilde i(\tilde \eta)$ is a loop in $\tilde N$ so $\eta = j(\tilde \eta)$ has $[\eta] = e \in \pi_1(N_0) =G$. Thus, $[\eta] \in H \subset \pi_1(K)$. This implies that $\tilde\eta$ is a loop, i.e., $a=b$. It is straightforward to check that $\tilde i$ is a closed map, using the fact that $K$ is compact and $\tilde p$ is a covering map. Therefore $\tilde i$ is proper.
		\end{proof}

		Note that $N$ is equipped with a Riemannian metric $g$ so that for any $p \in \tilde N$, each connected component of a level set of $f_{p}(x) =  d_{\tilde g}(x,p)$ has diameter $\leq C$, where $C=20L$ as given in Proposition \ref{prop:fill-to-Ur-w}. The embeddings $i,\tilde i$ induce metric structures on $K,\tilde K$, respectively.

			\begin{lemma}\label{lemm:bounded.K.dist.N.dist}
			For each $r>0$, there exists $R(r)>0$ such that for any $a,b \in\tilde K$ with $d_{\tilde N}(a,b) \leq r$, we have $d_{\tilde K}(a,b) \leq R(r)$. 
		\end{lemma}
		\begin{proof}
			Fix $x\in \tilde K$. By applying a deck transformation, we can assume that $d_{\tilde N}(a,x) \leq c_0$ (here $c_0$ only depends on $K$), so $d_{\tilde N}(b,x) \leq r+c_0$. Since $\tilde K$ is connected (and $\tilde i$ is proper), there exists $R=R(r+c_0)$ so that $d_{\tilde K}(a,x),d_{\tilde K}(b,x) \leq R$. The assertion follows from the triangle inequality. 
		\end{proof}

		In the following lemma, we will call a curve $\tilde \gamma : \RR\to\tilde K$ a \emph{line} if it minimizes length on compact subintervals relative to competitors in $\tilde K$. Note that such a curve is a geodesic in the sense of metric geometry, but not necessarily in the sense of Riemannian geometry, since it could stick to $\partial \tilde K$ in places. Similarly, we will call $\sigma' : [0,\infty)\to \tilde K$ an \emph{minimizing ray} if it minimizes length in the same sense. 
		
		\begin{lemma}
			There exists a line $\tilde \gamma$ in $\tilde K$. 
		\end{lemma}
		
		\begin{proof}
			Fix $p\in \tilde K$, and choose $p_j\in \tilde K$ diverging. Let $\sigma_j$ denote a curve that minimizes length in $\tilde K$ between $p$ and $p_j$. We assume that $\sigma_j$ is parametrized by unit speed. In particular, $\sigma_j$ is a $1$-Lipschitz map from an interval to $\tilde K$. 
			Consider an exhaustion of $\tilde K$ by nested compact sets containing $p$.
			Applying Arezel\`a--Ascoli in each compact set and taking a diagonal sequence we obtain that, after passing to a subsequence, $\sigma_i$ converges to a minimizing ray $\sigma': [0,\infty)\rightarrow \tilde K$. Since $G$ is the group of deck transformations of $\tilde K\to K$ acting transitively on $\tilde K$ and $K$ is compact, we can choose $t_i\rightarrow \infty$ and deck transformations $\Phi_i$ so that $d_{\tilde g}(p,\Phi_i(\sigma'(t_i)))$ is uniformly bounded. Then $\sigma_i'(t)=\Phi_i(\sigma'(t+t_i))$ subsequentially converges to a geodesic line $\sigma$ (using Arzel\`a--Ascoli again). 
		\end{proof}
		
		Parametrize the curve $\tilde \gamma$ so that $d_{\tilde K}(\tilde \gamma(a),\tilde \gamma(b)) = |a-b|$ (note that $d_{\tilde N}(\tilde\gamma(a),\tilde\gamma(b))$ might be smaller than $|a-b|$). Let $\gamma = \tilde{i} \circ \tilde{\gamma}$. Note that $\tilde\gamma$ is automatically proper in $\tilde K$, and thus Lemma \ref{lemm:tilde.i.proper} implies that $\gamma$ is proper in $\tilde N$.

		For each $R>0$, consider the open geodesic ball $B_{R}(\gamma(0)) \subset \tilde N$. Define parameters
		\begin{align*}t_{-}(R)& =\max\{t: \gamma(-\infty,t)\cap B_R(\gamma(0))=\emptyset\} \\
		t_+(R)& =\min\{t: \gamma(t,\infty)\cap B_R(\gamma_0)=\emptyset\}.\end{align*}
		Note that $t_\pm(R)\to \pm \infty$ as $R\to \infty$.

		Since $e(\tilde K)=1$, $\gamma(t_\pm(R))$ can be connected in $\tilde K \setminus B_R(\gamma(0))$. Because $\tilde N$ is simply connected, this implies that $\gamma(t_\pm(R))$ lie in the \emph{same} component of $\partial B_R(\gamma(0)) \subset \tilde N$ and thus 
		\[
		d_{\tilde N}(\gamma(t_-(R)),\gamma(t_+(R))) \leq C.
		\]
		On the other hand, we have 
		\[
		d_{\tilde K}(\gamma(t_-(R)),\gamma(t_+(R))) = |t_-(R)-t_+(R)| \to\infty.
		\]
		This contradicts Lemma \ref{lemm:bounded.K.dist.N.dist}. This completes the proof of Proposition \ref{prop:group.ends}. 
	\end{proof}

	\section{Generalization to the mapping problem} \label{sec:mapping}
	In this section we prove Theorem \ref{theo:mapping}.  The proof here is partly motivated by \cite[Section 5]{Gromov2020metrics}, where nonexistence of PSC metrics on certain noncompact manifold admitting a proper, distance decreasing map to a uniformly contractible manifold is established. We first observe that we may assume, without loss of generality, that $\pi_1(X)$ is infinite. Indeed, if $\pi_1(X)$ is finite, then the universal cover $\tilde X$ is compact and satisfies that $\pi_1(\tilde X)=\cdots=\pi_{n-2}(\tilde X)=0$. By the Hurewicz theorem, we have that $H_1(\tilde X)=\cdots=H_{n-2}(\tilde X)=0$. Poincar\'e duality further implies that $H_1(\tilde X)=\cdots=H_{n-1}(\tilde X)=0$, and hence $\tilde X$ is homeomorphic to $S^n$.

	We begin with the following general lemma. Note that it is tempting to try to lift a map of non-zero degree to the universal covers, but this map may not be proper (and hence the degree will not be well-defined).  We note that the construction of the appropriate cover is somewhat analogous to the construction of $\tilde K$ in Section \ref{section:homotopy.type}.
	\begin{lemma}
	    Suppose that $X,N$ are closed oriented manifolds and $f:N\to X$ has nonzero degree. Letting $\tilde X$ denote the universal cover of $X$, there exists a connected cover $\hat N \to N$ and a lift $\hat f : \hat N \to \tilde X$ so that $\hat f$ is proper and $\deg \hat f = \deg f$. 
	\end{lemma}
	\begin{proof}
	Choose a regular value $x\in X$ and set $f^{-1}(x) = \{z_1,\dots,z_k\}$. Consider $H : = \ker f_\# :\pi_1(N,z_1) \to \pi_1(X,x)$. Choose a covering space $p : \hat N \to N$ so that $\textrm{image } p_\# : \pi_1(\hat N,\hat z_1) \to \pi_1(N,z_1)$ is $H$. Below we will show that the map $f$ lifts to $\hat f: \hat N \to \tilde X$, and that $\hat f$ satisfies the assertions made above. 
	
   \emph{Non-compactness of $\hat N$:} We claim that $\hat N$ is non-compact. We first show that the image of $f_\#$ is a subgroup of $\pi_1(X,x)$ with finite index. Let $G=f_\#(\pi_1(N,z_1))$ and $\bar\pi: (\bar X, \bar x)\rightarrow (X,x)$ be a covering map so that image $(\bar{\pi})_\#: \pi_1(\bar X,\bar x)\rightarrow \pi_1(X,x)$ is $G$. The map $f$ lifts to a map $\bar f: (N,z_1)\to (\bar X,\bar x)$ such that $f=\bar \pi\circ \bar f$. Since $N$ is compact and $f$ is surjective, we see that $\bar X$ is compact. Hence we have $\deg f = \deg \bar \pi\cdot \deg \bar f$. It follows that $\deg \bar \pi$ is an integer factor of $\deg f$, and thus $G$ is a subgroup of $\pi_1(X,x)$ of finite index.
		
	The number of sheets of the covering map $p$ is the index of $H=p_\#(\pi_1(\hat N, \hat z_1))$ in $\pi_1(N,z_1)$. Since $H$ is a normal subgroup, this is equal to the number of elements of the group $\pi_1(N,z_1)/H$, which is isomomorphic to $G$ and thus of infinite order.  This implies that $\hat N$ is non-compact as claimed.

	\emph{Lifting the map $f$:} Consider $f \circ p : \hat N \to X$. Note that $(f\circ p)_\# :\pi_1(\hat N)\to \pi_1(X)$ is the zero map. Thus, we can lift $f\circ p$ to the universal cover of $X$:
	\[
	\xymatrix{
		(\hat N,\hat z_1) \ar[r]^{\hat f} \ar[d]_p & (\tilde X,\tilde x) \ar[d]^\pi \\
		(N,z_1) \ar[r]_f & (X,x)
	}
	\]
	Clearly, a loop in $N$ lifts to a loop in $\hat N$ if and only if it is in $H$ (recall that $H$ is normal).
	
	\emph{Counting lifts of pre-images:} We now claim that $\# ( \hat f^{-1}(\tilde x) \cap p^{-1}(z_j)) = 1$. To this end, suppose that $a,b \in  \hat f^{-1}(\tilde x) \cap p^{-1}(z_j)$. Choose a path $\hat \gamma$ in $\hat N$ connecting the two points. Then $\gamma = p\circ \hat\gamma$ is a loop in $N$ based at $z_j$. On the other hand, $\tilde\gamma := \hat f \circ \hat \gamma$ is a loop in $\tilde  X$ based at $\tilde x$. Since $e = \pi_\#[\tilde \gamma] = f_\#[\gamma]$, we thus see that $[\gamma]\in H$. This is a contradiction since this would imply that $\gamma$ lifts to a loop (as remarked above).

	\emph{Properness:} We now show that $\hat f$ is proper.
	  Assume that $\hat r_i\to\infty$ in $\hat N$ but $\hat f(\hat r_i) \to q$ in $\tilde X$. Since $N$ is compact, we can pass to a subsequence so that $p(\hat r_i) \to r\in N$. Then $\pi(q) = f(r)$. 
		
		Choose a contractible neighborhood $U\subset N$ with $r \in U$. By shrinking $U$ we can assume that $f(U)$ is contained in a contractible open set $W\subset X$. Then $\pi^{-1}(W)$ consists of disjoint copies of $W$. We can assume that $\hat f(\hat r_i)$ are all contained in the copy containing $q$.

		Assume that $p(\hat r_i) \in U$ for all $i$. Fix paths $\eta_i$ from $p(\hat r_i)$ to $r$ in $U$ and paths $\hat \gamma_i$ from $\hat r_1$ to $\hat r_i$ in $\hat N$. Then, 
		\[
		\alpha_i : = (\eta_i) * (p\circ \hat \gamma_i) * (-\eta_1)
		\]
		is a loop from $r$ to $r$. Lift $\alpha_i$ to $\hat\alpha_i$ a path in $\hat N$ that agrees with $\hat \gamma_i$ on that portion of $\hat\alpha_i$. Note that $\hat \alpha_i$ cannot be a loop for $i$ large, since the $\hat r_i$ are diverging. 
		
		We now consider $\tilde \alpha_i : = \hat f \circ \hat \alpha_i$ a path in $\tilde X$. By construction, $\tilde\alpha_i$ is a loop in $\tilde X$. This is a contradiction as before.
	
	\emph{Degree:} Finally, we check that $\deg \hat f=\deg f$. The lift $\tilde x$ is a regular point for $\hat f$ and we have seen that each element of $f^{-1}(x)$ lifts to a unique element of $\hat f^{-1}(\tilde x)$. But the local degree of $\hat f$ at each preimage $\hat z_i$ is the same as the degree of $f$ at the corresponding point $p(\hat z_i)$ (since $p$ is a covering map). 

    This completes the proof.
	\end{proof}
	
	\begin{remark}
		With some trivial modifications in the proof, a similar result holds for possibly nonorientable $X,N$ with a map $f: N\to X$ of nonzero mod $2$ degree.
	\end{remark}

%	Recall that 
	
%	We will show how to generalize the proof of Theorem \ref{theo:main} to conclude that some finite cover $\hat X$ is homotopy equivalent to $S^n$ or connected sums of $S^{n-1}\times S^1$. If $\pi_1(N)$ is finite, then its (finite) universal cover $\tilde N$ satisfies $\pi_1(\tilde N)=\cdots=\pi_{[\tfrac n2]}(\tilde N)=0$, and hence by the Hurewicz theorem, $H_1(\tilde N, \ZZ)=\cdots=H_{[\tfrac n2]}(\tilde N, \ZZ)=0$. By Poincar\'e duality, this implies that all of its homology groups vanish for degrees $1\leq k\leq n-1$. Therefore $\tilde N$ is homeomorphic to $S^n$. Thus, we assume that $\pi_1(N)$ is infinite.
	
%	It is tempting to combine the proof of Theorem \ref{theo:main} with the arguments in \cite[Section 5]{Gromov2020metrics} where we consider the lifted map $\tilde f: \tilde N\to \tilde X$ between universal covers. However, in general such a lifting may not be proper (and hence the mapping degree is not well defined). Instead, we lift $f$ to a suitable covering space of $N$. We note that the construction of such covering is somewhat analogous to the construction of $\tilde K$ in Section \ref{section:homotopy.type}.

	Using the lifted map $\hat f$ we can now follow \cite[Section 5]{Gromov2020metrics} to show that the conclusion of Corollary \ref{coro:filling-PSC} holds in the setting of Theorem \ref{theo:mapping}. %Fix a metric $g_X$ on $X$ and by scaling, we can assume that $f:(N,g)\to (X,g_X)$ is distance decreasing. 

	\begin{lemma}\label{lemm:mapping.fill.rad}
		Let $X,N$ be oriented Riemannian manifolds, $f:(N,g)\to (X,g_X)$ with $f$ distance decreasing and $\deg f \neq 0$. Assume that $N$ admits a metric of positive scalar curvature and that either: $n=4$ and $\pi_2(X)=0$ or $n=5$ and $\pi_2(X)=\pi_3(X)=0$.  
		
		Then, there exists $L>0$ with the following property.
		If $\Sigma_{n-2}$ is an $(n-2)$-dimensional null-homologous cycle in the universal cover
		$\tilde{X}$ of $X$, then the cycle $deg(f) \Sigma_{n-2}$ can be filled inside
		$B_L(\Sigma_{n-2})$.
	\end{lemma}
	\begin{proof}
		We consider $n=5$ since the $n=4$ case is similar (but simpler). By scaling we can assume that $(N,g)$ has scalar curvature $R\geq 1$. As in Corollary \ref{coro:filling-PSC}, $H_2(\tilde X,\ZZ) = H_3(\tilde X,\ZZ) = 0$. 
		
		By assumption, $\Sigma_{3} = \partial \Sigma_{4}$ in $\tilde X$ for some chain $\Sigma_{4}$. Up to a small perturbation, we can assume that $\hat f$ is transversal to $\Sigma_3$ and $\Sigma_{4}$. Set $\hat\Sigma_{4}:=\hat f^{-1}(\Sigma_{4})$ and similarly $\hat\Sigma_{3} = \partial\hat\Sigma_{4}$. Note that $\hat\Sigma_{3}$ is null-homologous in $\hat N$ (by construction). Using Theorem \ref{thm:statement-aspherical} we can find $\hat\Sigma_3'\subset B_{L_0}(\hat\Sigma_3)$ homologous to $\hat\Sigma_3$ as well as $3$-chains $\hat U_1,\dots,\hat U_m$ with $\diam(U_j)\leq L_0$ and $2$-cycles $\{\hat \Gamma_j^l : j =1,\dots,m, l=1,\dots,k(j)\}$ with $\diam(\hat \Gamma_j^l) \leq L_0$ and so that
	        \[
	        \hat\Sigma_3' = \sum_{j=1}^m \hat U_j
	        \]
	        and
	        \[
	        \partial \hat U_j = \sum_{l=1}^{k(j)} \hat\Gamma_j^l,
	        \]
	        for each $j=1,\dots,m$, where both equalities hold as chains (not just in homology).  Finally, there is an integer $q$ and a function 
	        \[
	        u : \{(j,l) : j =1,\dots,m, l=1,\dots,k(j)\} \to \{1,\dots,q\}
	        \]
	        so that for $r \in \{1,\dots,q\}$, we have
	        \[
	        \diam\left(\cup{(j,l) \in u^{-1}(r)} \hat\Gamma_j^l \right) \leq L_0
	        \]
	        and
	        \[
	        \sum_{(j,l) \in u^{-1}(r)} \hat\Gamma_j^l = 0
	        \]
	        as $2$-cycles for $r\in\{1,\dots,q\}$. 
	        
    	   Denote by $\Sigma_3'$ the $3$-cycle in $\tilde X$ obtained by pushing $\hat\Sigma_3'$ forward by the map $\hat f$ and similarly for $U_j$ and $\Gamma_j^l$. 
    	   
    	   Since $\hat f$ is transversal to $M_3$, it is easy to check that $\deg \hat f|_{M_3} = \deg \hat f$. Hence, $\hat f_\#([\hat\Sigma_3']) = (\deg \hat f)[\Sigma_3]$. Moreover, since $f$ (and thus $\hat f$) was assumed to be distance decreasing, we see that $d_{(\tilde X,g_{\tilde X})}(\Sigma_3,\Sigma_3')\leq L_0$. As such, it suffices to bound $\Sigma_3'$ in a controlled neighborhood. 
    	   
    	   To do so, we follow the argument used in Corollary \ref{coro:filling-PSC}. Because $\diam(\hat\Gamma_j^l) \leq L_0$,  we can use Proposition \ref{prop:FF-iso-cover} to find a $3$-chain $\tilde\Gamma_l^j$ with $\diam(\tilde\Gamma_j^l)\leq R(L_0)$ and $\partial\tilde\Gamma_l^j = \Gamma_l^j$ and then a $4$-chain $\tilde U_j$ with
    	   \[
    	   \partial\tilde U_j = U_j - \sum_{l=1}^{k(j)} \tilde \Gamma_j^l
    	   \]
    	   and $\diam(\tilde U_j) \leq R(L_0+2R(L_0))$. Thus,
    	   \[
    	   \Sigma_3' = \sum_{j=1}^m \partial\tilde U_j + \sum_{r=1}^q \sum_{(j,l)\in u^{-1}(r)} \tilde \Gamma_j^l. 
     	   \]
     	   and 
     	   \[
            \diam\left(\sum_{(j,l)\in u^{-1}(r)} \tilde \Gamma_j^l \right) \leq 2R(L_0)+L_0.
     	   \]
     	   We can thus complete the proof as in Corollary \ref{coro:filling-PSC}.
		%\newpage

		%To do so, we simply observe that we can push-forward the slice-and-dice decomposition of $\hat\Sigma_3'$ as obtained in \cite[Section 6]{ChodoshLi2020generalized} to obtain a decomposition of $\Sigma_3$ into $3$-chains with diameter bounded by $10\pi$, whose boundaries can be filled in an $R$-neighborhood (for $R$ large but only depending on $(X,g)$) thanks to Proposition \ref{prop:FF-iso-cover}. Arguing as in \cite[Section 6]{ChodoshLi2020generalized}, we thus find that $\hat\Sigma_3$ can be filled in a $2R$-neighborhood (taking $R$ larger if necessary, but still only depending on $(X,g)$). Thus, the assertion follows after taking $L=2R+4\pi$.
	\end{proof}
	
	Granted Lemma \ref{lemm:mapping.fill.rad}, Theorem \ref{theo:mapping} follows. Indeed, in order to prove the Urysohn width estimate Proposition \ref{prop:fill-to-Ur-w}
	it is enough to assume that the filling radius estimate holds for a multiple 
	$deg(f) \Sigma_{n-2}$ of every cycle $\Sigma_{n-2}$. The rest of the proof of Theorem \ref{theo:mapping} proceeds exactly
	as the proof of Theorem \ref{theo:main}.

	\bibliography{bib}
	\bibliographystyle{amsplain}
	
	%\printbibliography
	
\end{document}